\documentclass[a4paper,12pt]{amsart}

\usepackage[latin2]{inputenc}
\usepackage{amsmath}
\usepackage{amscd,amssymb,amsthm}

\begin{document}

\title{About the Non-Integer Property of Hyperharmonic Numbers}
\author{Istv\'an Mez\H{o}}
\address{Department of Algebra and Number Theory, Institute of Mathematics, University of Debrecen, Hungary}
\email{imezo@math.klte.hu}
\urladdr{http://www.math.klte.hu/algebra/mezo.htm}

\newcommand{\NN}{\mathbb{N}}
\newcommand{\ZZ}{\mathbb{Z}}
\newcommand{\QQ}{\mathbb{Q}}
\newcommand{\lcm}{\mathop{\textup{lcm}}\nolimits} 
\newcommand{\ord}{\mathop{\textup{Ord}}\nolimits} 

\newtheorem{Theorem}{Theorem}
\newtheorem{Lemma}[Theorem]{Lemma}
\newtheorem{Corollary}[Theorem]{Corollary}
\newtheorem{Conjecture}[Theorem]{Conjecture}
\theoremstyle{definition}
\newtheorem{Example}[Theorem]{Example}
\newtheorem{Problem}[Theorem]{Problem}

\begin{abstract}It was proven in 1915 by Leopold Theisinger that the $H_n$ harmonic numbers are never integers. In 1996 Conway and Guy have defined the concept of hyperharmonic numbers. The question naturally arises: are there any integer hyperharmonic numbers? The author gives a partial answer to this question and conjectures that the answer is ``no''.
\end{abstract}

\maketitle

The $n$-th harmonic number is the $n$-th partial sum of the harmonic series:
\[H_n=\sum_{k=1}^n\frac{1}{k}.\]

Conway and Guy in \cite{CG} defined the harmonic numbers of higher orders, also known as the hyperharmonic numbers: $H_n^{(1)}:=H_n$, and for all $r>1$ let
\[H_n^{(r)}=\sum_{k=1}^n H_k^{(r-1)}\]
be the $n$-th harmonic number of order $r$.
These numbers can be expressed by binomial coefficients and ordinary harmonic numbers:
\[H_n^{(r)}=\binom{n+r-1}{r-1}(H_{n+r-1}-H_{r-1}).\]

The prominent role of these numbers has been realized recently in combinatory.
The $[\begin{smallmatrix}n\\k\end{smallmatrix}]_r$ $r$-Stirling number is the number of the permutations of the set $\{1,\dots,n\}$ having $k$ disjoint, non-empty cycles, in which the elements $1$ through $r$ are restricted to appear in different cycles.

In \cite{BGG} one can find the following interesting equality:

\[H_n^{(r)}=\frac{[\begin{smallmatrix}n+r\\r+1\end{smallmatrix}]_r}{n!}.\]

Let us turn our attention to the main question of this paper. It is known that any number of consecutive terms not necessarily beginning with 1 will never sum to an integer (see \cite{MW}). As a corollary, we get that the $H_n$ harmonic numbers are never integers ($n>1$). Theisinger proved this latter result directly in 1915 \cite{T}. The question appears obviously: are there any integer hyperharmonic numbers?

Theisinger's main tool was the $2$-adic norm. We give a short summary of his method.
Every rational number $x\neq 0$ can be represented by $x=\frac{p^\alpha r}{s}$, where $p$ is a fixed prime number, $r$ and $s$ are relative prime integers to $p$. $\alpha$ is a unique integer. We can define the $p$-adic norm of $x$ by
\[|x|_p=p^{-\alpha},\mbox{ and let }|0|_p=0.\]
This norm fulfills the properties of the usual norms, namely
\[|x|_p=0\Longleftrightarrow\,x=0,\]
\[|xy|_p=|x|_p|y|_p\quad(x,y\in\QQ),\]
\[|x+y|_p\le|x|_p+|y|_p\quad(x,y\in\QQ).\]
Furthermore, the so-called strong triangle inequality also holds:
\[|x+y|_p\le\max\{|x|_p,|y|_p\}(\le|x|_p+|y|_p).\]

We shall use the following property of integer numbers:
\[x\in\ZZ\Longrightarrow x=p^\alpha r\Longrightarrow|x|_p=\frac{1}{p^\alpha}\le 1,\]
where $r$ and the prime $p$ are relative prime integers. This means that if the $p$-norm of a rational $x$ is greater than 1 then $x$ is necessarily non-integer.

Let us introduce the order of a natural number $n$: if $2^m\le n<2^{m+1}$, then $\ord_2(n):=m$.
It is obvious that $\ord_2(n)=\left\lfloor\log_2(n)\right\rfloor$.

\begin{Theorem}
\[|H_n|_2=2^{\ord_2(n)}\quad(n\in\NN),\]
that is -- by our observation above -- $H_n$ is never integer.
\end{Theorem}

\begin{proof}First, let $n$ be even.
Since $|x|_2=|-x|_2$ for all $x\in\QQ$, by the strong triangle inequality we get
\[\max\left\{|H_n|_2,|1|_2\right\}=\max\left\{|H_n|_2,|1|_2,\left|\frac{1}{3}\right|_2,\left|\frac{1}{5}\right|_2,\dots,\left|\frac{1}{n-1}\right|_2\right\}\ge\]
\[\ge\left|H_n-1-\frac{1}{3}-\frac{1}{5}-\cdots-\frac{1}{n-1}\right|_2=\]
\[=\left|\frac{1}{2}+\frac{1}{4}+\cdots+\frac{1}{n-2}+\frac{1}{n}\right|_2=\]
\[=\left|\frac{1}{2}\right|_2\left|1+\frac{1}{2}+\frac{1}{3}+\cdots+\frac{1}{n/2}\right|_2=2\left|H_{n/2}\right|_2.\]

If $n$ is odd, the situation is the same:
\[\max\left\{|H_n|_2,|1|_2\right\}\ge 2\left|H_{(n-1)/2}\right|_2.\]
The reader may verify it.

So we get that the 2-adic norm of the harmonic numbers is monotone increasing. Since $|H_2|_2=\left|\frac{3}{2}\right|_2=2$, the 2-adic norm of all the harmonic numbers are greater than 1. As a corollary, this means that the harmonic numbers are not integers because of the property of the 2-adic norm mentioned above.

We can continue the calculations on $H_{n/2}$ (or on $H_{(n-1)/2}$) instead of $H_n$. For instance let us consider that $n/2$ is even. Then the method described above gives that
\[|H_{n/2}|_2\ge\left|\frac{1}{2}\right|_2|H_{n/4}|_2=2|H_{n/4}|_2.\]
This and the previous estimation implies that
\[|H_n|_2\ge\left|\frac{1}{2}\right|_2|H_{n/2}|_2\ge\left|\frac{1}{2}\right|_2\left|\frac{1}{2}\right|_2|H_{n/4}|_2=4|H_{n/4}|_2.\]
And so on. If $n/2$ is odd then we choose $(n/2-1)/2$ instead of $n/4$. We can perform these steps exactly $\ord_2(n)$ times.

After all, we shall have the following:
\[|H_n|_2\ge \left|\frac{1}{2^{\ord_2(n)}}\right|_2|H_1|_2=2^{\ord_2(n)}.\]

On the other hand,
\[|H_n|_2\le\max\left\{|1|_2,\left|\frac{1}{2}\right|_2,\cdots,\left|\frac{1}{n}\right|_2\right\}=\left|\frac{1}{2^{\ord_2(n)}}\right|_2=2^{\ord_2(n)},\]
because the greatest 2-power occuring between $1$ and $n$ is $\ord_2(n)$.

The inequalities detailed above give the statement.
\end{proof} 

A different approach can be found in \cite{GKP}, \cite{BR} and in their references. The next proof comes from these sources.

\begin{proof}Let us fix the order of $n$, i.e.: $\ord_2(n):=m$. This implies that the denominator of $\frac{2^{m-1}}{n}$ is odd, unless $n=2^m$. We get that the number
\[2^{m-1}H_n-\frac{1}{2}\]
can be represented by the sum of rationals with odd denominators. For example,
\[2^{m-1}H_n-\frac{1}{2}=\frac{a_1}{b_1}+\cdots+\frac{a_s}{b_s}=\frac{c}{\lcm(b_1,\dots,b_s)},\]
where $b_i$ is odd for all $i=1,\dots,s$. It means that $b:=\lcm(b_1,\dots,b_s)$ is odd. The last formula gives the result
\[H_n=\frac{\frac{c}{b}+\frac{1}{2}}{2^{m-1}}=\frac{2c+b}{2^m b}.\]
\end{proof}

Let us turn our attention to hyperharmonic numbers. We need a lemma which can be found in \cite{BS}:

\begin{Lemma}
\[|n!|_p=p^{(A_p(n)-n)/(p-1)},\]
where $A_p(n)$ is the sum of the digits of the $p$-adic expansion of $n$.
\end{Lemma}

\begin{Example} Let $p=2$ and $n=11$. Then $n=1011_2$, that is, $A_2(n)=3$.
\[|n!|_2=|39916800|_2=|256\cdot155925|_2=|256|_2|155925|_2=|2^8|_2\cdot1=2^{-8}.\]
We can apply the lemma: $A_2(n)-n=3-11=-8$, whence $|n!|_2=2^{-8}$.
\end{Example}

\begin{Theorem} If $\ord_2(n+r-1)>\ord_2(r-1)$ then
\[|H_n^{(r)}|_2=2^{A_2(n+r-1)-A_2(n)-A_2(r-1)+\ord_2(n+r-1)},\]
else
\[|H_n^{(r)}|_2=2^{A_2(n+r-1)-A_2(n)-A_2(r-1)+\max\left\{\left|\frac{1}{r}\right|_2,\left|\frac{1}{r+1}\right|_2,\dots,\left|\frac{1}{n+r-1}\right|_2\right\}}.\]
\end{Theorem}

\begin{proof}
\[\left|H_n^{(r)}\right|_2=\left|\binom{n+r-1}{r-1}(H_{n+r-1}-H_{r-1})\right|_2=\]
\[=\left|\binom{n+r-1}{r-1}\right|_2\left|\frac{a}{2^{\ord_2(n+r-1)}b}-\frac{c}{2^{\ord_2(r-1)}d}\right|_2=\]
\[=\left|\binom{n+r-1}{r-1}\right|_2\left|\frac{2^{\ord_2(r-1)}ad-2^{\ord_2(n+r-1)}bc}{2^{\ord_2(n+r-1)+\ord_2(r-1)}bd}\right|_2.\]
Because of the condition $\ord_2(n+r-1)>\ord_2(r-1)$ we get
\[\left|H_n^{(r)}\right|_2=\left|\binom{n+r-1}{r-1}\right|_2\left|\frac{ad-2^{\ord_2(n+r-1)-\ord_2(r-1)}bc}{2^{\ord_2(n+r-1)}bd}\right|_2.\]
Since the nominator is odd, we get the following:
\[\left|\frac{ad-2^{\ord_2(n+r-1)-\ord_2(r-1)}bc}{2^{\ord_2(n+r-1)}bd}\right|_2=2^{\ord_2(n+r-1)}.\]
To compute the $2$-adic norm of the binomial coefficient, we use the previous lemma.
\[\left|\binom{n+r-1}{r-1}\right|_2=\left|\frac{(n+r-1)!}{(r-1)!n!}\right|_2=\]
\[=\frac{2^{A_2(n+r-1)-n-r+1}}{2^{A_2(r-1)-r+1}2^{A_2(n)-n}}=2^{A_2(n+r-1)-A_2(n)-A_2(r-1)}.\]
This, and the previous equality give the result with respect to the condition $\ord_2(n+r-1)>\ord_2(r-1)$.

Let us fix an arbitrary $n$ for which $\ord_2(n+r-1)=\ord_2(r-1)$.
\[H_{n+r-1}-H_{r-1}=\frac{1}{r}+\frac{1}{r+1}+\cdots+\frac{1}{n+r-1}.\]
Let us substract all of the fractions with odd denominators. Then we can take $\frac{1}{2}$ out of the remainder and continue the recursive method described in the first proof of Theorem 1. We can make such substraction steps
\[\max\left\{\left|\frac{1}{r}\right|_2,\left|\frac{1}{r+1}\right|_2,\dots,\left|\frac{1}{n+r-1}\right|_2\right\}\]
times. The result:
\[|H_{n+r-1}-H_{r-1}|_2\ge\max\left\{\left|\frac{1}{r}\right|_2,\left|\frac{1}{r+1}\right|_2,\dots,\left|\frac{1}{n+r-1}\right|_2\right\}.\]

On the other hand, by the strong triangle inequality
\[|H_{n+r-1}-H_{r-1}|_2\le\max\left\{\left|\frac{1}{r}\right|_2,\left|\frac{1}{r+1}\right|_2,\dots,\left|\frac{1}{n+r-1}\right|_2\right\}.\]
\end{proof}

\begin{Corollary}The sum of the harmonic numbers cannot be integer:
\[H_1+H_2+\cdots+H_n\not\in\NN\quad(n>1).\]
Or, which is the same,
\[H_n^{(2)}\not\in\NN\quad(n>1).\]
\end{Corollary}

\begin{proof}$H_1+H_2+\cdots+H_n=H_n^{(2)}$. The condition with respect to the order of $n$ and $r$ holds because $\ord_2(n+2-1)>\ord_2(2-1)=0$ for all $n\ge1$. Furthermore,
\[\left|H_n^{(2)}\right|_2=2^{A_2(n+1)-A_2(n)-A_2(1)+\ord_2(n+1)}.\]
Let $m=\ord_2(n+1)$. Our goal is to minimize the power of $2$. $\ord_2(n+1)=m$ implies that $n+1<2^{m+1}$, therefore $1\le A_2(n+1)\le m+1$ and $1\le A_2(n)\le m$. The minimum in the power is taken when $A_2(n)=m$ and $A_2(n+1)=1$. It is possible if and only if $n=2^m-1$. In this case
\[A_2(n+1)-A_2(n)-A_2(1)+\ord_2(n+1)=1-m-1+m=0.\]

We get that if $n\neq 2^m-1$ for some $m$, then $|H_n^{(2)}|_2>1$, that is, $H_n^{(2)}\not\in\NN$. On the other hand, let us assume that $n$ has the form $2^m-1$. This implies that
\[H_n^{(2)}=\binom{n+2-1}{2-1}(H_{n+2-1}-H_{2-1})=\]
\[=(n+1)(H_{n+1}-1)=2^m\left(\frac{a}{2^{\ord_2(n+1)}b}-1\right)=\frac{a}{b}-2^m\not\in\NN.\]
\end{proof}

One can easily prove the following, using the method in the previous proof.

\begin{Corollary}$H_n^{(3)}\not\in\NN$ for all $n>1$.
\end{Corollary}

As we can see, the method to prove the non-integer property of harmonic numbers does not work for hyperharmonic numbers, because there are $n$ and $r$ integers for which $|H_n^{(r)}|_2=1$. In spite of this fact, we believe that Theisinger's theorem holds for all hyperharmonic numbers, too.
\begin{Conjecture}
None of the hyperharmonic numbers can be integers ($r,n\ge 2$).
\end{Conjecture}

\begin{Example}We demonstrate that the theorem described above simplifies the calculation of the $2$-norm of hyperharmonic numbers.

For instance, 
\[H_{18}^{(8)}=\binom{18+8-1}{8-1}(H_{18+8-1}-H_{8-1})=\]
\[=480700\left(\frac{34052522467}{8923714800}-\frac{363}{140}\right)=\frac{10914604807}{18564}.\]
Since $|18564|_2=2^{-2}$, we get that $|H_{18}^{(8)}|_2=2^2=4$.

On the other hand, $A_2(18+8-1)=A_2(16+8+1)=3$, $A_2(8-1)=A_2(4+2+1)=3$, $A_2(18)=A_2(16+2)=2$ and $\ord_2(18+8-1)=\ord_2(16+9)=4$. By theorem 5,
\[|H_{18}^{(8)}|_2=2^{3-3-2+4}=2^2=4.\]
\end{Example}

Finally, we pose an interesting question:

\begin{Problem}For which $n_1\neq n_2$ and $r_1\neq r_2$ does the equality
\[H_{n_1}^{(r_1)}=H_{n_2}^{(r_2)}\]
stand?
\end{Problem}

\end{document}